\documentclass[11pt]{article}
\usepackage{amsmath, amsthm, amscd, amsfonts}
\newtheorem{thm}{Theorem}[section]
\newtheorem{exam}[thm]{Example}
\newtheorem{defn}{Definition}[section]
\newtheorem{rem}{Remark}[section]

\begin{document}{}
\baselineskip	9mm
\vspace*{8mm}
\begin{center}{
\Large\bf	The Lumer-Phillips Theorem For Two--parameter $C_0$--semigroups}\\
\vspace{4mm}
{\bf	Rasoul Abazari\footnote{Corresponding Author E-mail:r.abazari@iauardabil.ac.ir,\\ rasoolabazari@gmail.com}, Assadollah Niknam, Mahmoud Hassani}\\
\vspace{4mm}{\normalsize{\em \small Department of
Mathematics, Faculty of Sciences, Mashhad Branch, Islamic Azad University, Mashhad, Iran.}\\
\vspace*{2mm}
}
\end{center}
\begin{quotation}
\vspace{4mm}	\small
{\noindent\bf	Abstract:}
In this paper we extend the Lumer-Phillips theorem to the context of two--parameter $C_0$--semigroup of contractions. That is, we characterize
the infinitesimal generators of two--parameter $C_0$--semigroups of contractions. Conditions on the behavior of the resolvent of operators, which are necessary and sufficient for the pair of operators to be the infinitesimal generator of a $C_0$--semigroup of contractions are given.
\vspace*{2mm}
\end{quotation}
{\normalsize{
\section{Preliminaries}
The semigroups of operators have several application in areas of applied mathematics such as
prediction theory and random fields. This theory is useful to describe the time evolution of physical system
in quantum field theory, statistical mechanic and partial differential equations\cite{Van.Castern}, \cite{A. Niknam}.

In this section, we state some definitions and theorems as preliminaries to describe the main results.
We start by state the definition of two--parameter semigroups.

\begin{defn}
Let $X$ be a Banach space. By a two--parameter semigroup of operators we mean a function
$T:\mathbb{R}_+\times\mathbb{R}_+\longrightarrow B(X)$ with the following properties;

i) $T(0, 0)=I$

ii) $T(s+s', t+t')=T(s, t)T(s', t')$\\
If $(s,t)\longrightarrow T(s,t)x$ is continuous for all $x\in X$, then it is called strongly continuous and
if $(s,t)\longrightarrow T(s,t)x$ is norm continuous, then it is called uniformly continuous.
\end{defn}
A strongly continuous semigroup of bounded linear operators on $X$ will be called a semigroup of
class $C_0$ or simply $C_0$--semigroup.

Let $T(s,t)$ be any two--parameter semigroup, if we consider $u(s)$ and $v(t)$ as below,
$$u(s)=T(s,0) \ \ \ , \ \ \ \ v(t)=T(0,t)$$
then the semigroup property of $T$ implies that $T(s,t)=u(s)v(t)$ and $T(s,t)$ is strongly (resp. uniformly)
continuous if and only if $u(s)$ and $v(t)$ are strongly (resp. uniformly) continuous as one--parameter semigroup.

If $A_1$ and $A_2$ are infinitesimal generators of $u(s)$ and $v(t)$ respectively,
then we will thinks of the pair $(A_1, A_2)$ as infinitesimal generator of $T(s,t)$. For more details on the such generators, see \cite{Abazari}

\begin{thm}
\cite{A.Pazy}, Let $T(t)$ be $C_0$--semigroup, there exist constants $\omega\geq0$ and $M\geq1$ such that
$$\|T(t)\|\leq Me^{\omega t}, \ \ \ \ \text{for}\  \ 0\leq t<\infty.$$
\end{thm}

Let $T(s,t)$ be a $C_0$--semigroup,
Since $T(s,t)=u(s)v(t)$ and $u(s), v(t)$ are $C_0$--semigroup in the manner of one--parameter,
then by the previous theorem there exist constants $\omega_1, \omega_2\geq0$ and $M_1, M_2\geq1$such that
$$\|u(s)\|\leq M_1e^{\omega_1s},$$
$$\|v(t)\|\leq M_2e^{\omega_2t}.$$
Let $M=M_1M_2$, then we have
\begin{eqnarray*}
\begin{split}
\|T(s,t)\|&=\|u(s)v(t)\|\leq\|u(s)\|\|v(t)\|\\
&\leq M_1M_2e^{\omega_1s}e^{\omega_2t}=Me^{\omega_1s+\omega_2t}.
\end{split}
\end{eqnarray*}

\begin{defn}
If $\omega_1=\omega_2=0,$ then $T(s,t)$ is called uniformly bounded and if
moreover $M=1$ it is called semigroup of contractions.
\end{defn}

Recall that if $A$ is a linear, not necessarily bounded operator in Banach space $X$, the resolvent
set $\rho(A)$ of $A$ is the set of all complex numbers $\lambda$ for which $\lambda I-A$ is invertible
i.e, $(\lambda I-A)^{-1}$ is a bounded linear operator in $X$. The family $R(\lambda, A)=(\lambda I-A)^{-1}$,
 $\lambda\in \rho(A)$ of bounded linear operators is called the resolvent of $A$.\\
Let $X$ be a Banach space and $X^*$ be its dual. $<x^*, x>$ or $<x, x^*>$ denotes the value of
$x^*\in X^*$ at $x\in X$. For every $x\in X$ we define the duality set $F(x)\subset X^*$ by
$$F(x)=\{x^* : \ \ x^*\in X^* \ \ and \ \ <x^*,x>=\|x\|^2=\|x^*\|^2\}.$$
From the Hahn-Banach theorem it follows that $F(x)\neq\phi$ for every $x\in X.$

A linear operator $A$ is dissipative if for every $x\in D(A)$, the domain of $A$, there is a $x^*\in F(x)$
such that $Re<Ax, x^*>\leq0.$


\section{Main Results}

We state first the following useful theorems which can be found for example in \cite{A.Pazy}, \cite{Engel}.
\begin{thm}
A linear operator $A$ is dissipative if and only if,
$$\|(\lambda I-A)x\|\geq\lambda\|x\|\  \  \text{for all}\  \  x\in D(A) \ \ \text{and}\  \  \lambda>0.$$

\end{thm}

\begin{thm}
For a dissipative operator $(A, D(A))$ the following properties hold.

i) $\lambda-A$ is injective for all $\lambda>0$ and
$$\|(\lambda-A)^{-1}z\|\leq\frac{1}{\lambda}\|z\|,$$
for all $z$ in the range $R(\lambda-A)=(\lambda-A)D(A)$.

ii) $\lambda-A$ is surjective for some $\lambda>0$ if and only if it is surjective for each $\lambda>0$. In that case, one has $(0, \infty)\subset\rho(A)$.

ii) $A$ is closed if and only if the range $R(\lambda-A)$ is closed for some (hence all) $\lambda>0$.

iv) If $R(A)\subseteq\overline{D(A)}$ then $A$ is closable. Its closure $\overline{A}$ is again dissipative and satisfies $R(\lambda-\overline{A})=\overline{R(\lambda-A)}$ for all $\lambda>0$.
\end{thm}

The following theorem can be found in \cite{Engel}.
\begin{thm}
A pair $(A_1, A_2)$ of operators with domain in $X$ is infinitesimal generator of $C_0$--two--parameter semigroup $T(s,t)$
satisfying $\|T(s,t)\|\leq M_0e^{\omega s+\omega' t},$ for some $M_0\geq1, \omega, \omega'>0,$ if and only if

(i) $A_1$ and $A_2$ are closed and densely defined operators and
$$R(\lambda', A_2)R(\lambda, A_1)=R(\lambda, A_1)R(\lambda', A_2),$$
for each $\lambda\geq\omega, \lambda'\geq\omega'.$

(ii) The resolvent sets $\rho(A_1)$ and $\rho(A_2)$ contain $[\omega, \infty)$ and $[\omega', \infty)$,
respectively and there is some $M\geq1$ such that,
$$\|R(\lambda, A_1)^n\|\leq\frac{M}{(Re\lambda-\omega)^n},$$
$$\|R(\lambda', A_2)^n\|\leq\frac{M}{(Re\lambda'-\omega')^n},$$
where $Re\lambda\geq\omega$ and $Re\lambda'\geq\omega'.$

\end{thm}

Now we state extended Lumer--Phillips theorem as follows.

\begin{thm}
Let $A_1$ and $A_2$ are linear operators with dense domain in $X$.

(a) If $A_1$ and $A_2$ are dissipative and $A_2$ be bounded and there exist $\lambda_1, \lambda_2>0$ such that $R(\lambda_1I-A_1)=R(\lambda_2I-A_2)=X$ and $A_1A_2=A_2A_1$ then $(A_1, A_2)$ is the infinitesimal generator of a $C_0$--semigroup of contractions on $X$.

(b) If $(A_1, A_2)$ is the infinitesimal generator of $C_0$--semigroup of contractions on $X$, then
$$R(\lambda I-A_1)=R(\lambda I-A_2)=X, \ for \ all \ \lambda>0,$$
$A_1$ and $A_2$ are dissipative. Moreover for every $x\in D(A_1), y\in D(A_2), x^*\in F(x)$ and $y^*\in F(y)$, we have
$$Re<A_1x,x^*>\leq0,$$
and
$$Re<A_2y,y^*>\leq0.$$
\end{thm}

\begin{proof}

(a) Since $R(\lambda_1I-A_1)=R(\lambda_2I-A_2)=X,$ then by theorem 2.2, $R(\lambda I-A_1)=R(\lambda I-A_2)=X$ for every $\lambda>0$.
 Therefore  $\rho(A_1)\supset [0, \infty), \rho(A_2)\supset [0, \infty)$ and by theorem 2.1 we have,
  $\|R(\lambda I, A_1)\|\leq\lambda^{-1}$ and $\|R(\lambda I, A_2)\|\leq\lambda^{-1}.$

  On the other hand, let $\lambda, \lambda'>0$. Hence by theorem 2.1 we have that $(\lambda I-A_1)^{-1}$ and $(\lambda' I-A_2)^{-1}$
  exist. By the assumption $A_1A_2=A_2A_1$ hence
  \begin{eqnarray}\label{eq1}
  (\lambda I-A_1)(\lambda'I-A_2)=(\lambda' I-A_2)(\lambda I-A_1),
  \end{eqnarray}
  Also $(\lambda I-A_1)D(A_2)=X$. Since $A_2$ is bounded therefore
  $$(\lambda' I-A_2)(\lambda I-A_1)D(A_1)=(\lambda' I-A_2)X=X.$$

  Now let $y\in X$, so there is some $x\in D(A_1)$ such that
  \begin{eqnarray*}
\begin{split}
y&=(\lambda'I-A_2)(\lambda I-A_1)x\\
&=(\lambda I-A_1)(\lambda' I-A_2)x,
\end{split}
\end{eqnarray*}
last equality holds from (\ref{eq1}).
Therefore we have
$$R(\lambda', A_2)R(\lambda, A_1)y=x=R(\lambda, A_1)R(\lambda', A_2)y,$$
and also,
$$R(\lambda', A_2)R(\lambda, A_1)=R(\lambda, A_1)R(\lambda', A_2).$$
By theorem 2.3, we conclude that $(A_1, A_2)$ is the infinitesimal generator of a $C_0$--two--parameter semigroup of
contractions on $X$.

(b) If $(A_1,A_2)$ is the infinitesimal generator of a $C_0$--two--parameter semigroup $\{W(s,t)\}$ of contractions on $X$.
Then by theorem 2.3 part (ii), $[0,\infty)$ is contained in $\rho(A_1)$ and $\rho(A_2)$, therefore
$$R(\lambda I-A_1)=R(\lambda I-A_2)=X, \ \ \ for\  all \ \  \lambda>0.$$

For prove the dissipatedness of $A_1$ and $A_2$, following the proof of theorem 4.3 in \cite{A.Pazy} for the case of one--parameter, let $x\in D(A_1), \ y\in D(A_2), \ x^*\in F(x)$ and $y^*\in F(y)$.
Hence
$$|<W(s,0)x,x^*>|\leq\|W(s,0)x\|\|x^*\|\leq\|x\|^2,$$
$$|<W(0,t)y,y^*>|\leq\|W(0,t)y\|\|y^*\|\leq\|y\|^2,$$
and therefore
$$Re<W(s,0)x-y,x^*>=Re<W(s,0)x,x^*>-\|x\|^2\leq0,$$
$$Re<W(0,t)y-y,y^*>=Re<W(0,t)y,y^*>-\|y\|^2\leq0.$$
Dividing above states to $s$ and $t$ respectively and letting $s$ and $t$ to zero, yield
$$Re<A_1x,x^*>\leq0,$$
and
$$Re<A_2y,y^*>\leq0.$$

These hold for every $x^*\in F(x)$ and $y^*\in F(y)$ and complete the proof.

\end{proof}

The following example shows that there is a Banach space and operators satisfying conditions in theorem 2.4.
\begin{exam}
Suppose $X$ be a set of functions on $\mathbb{R}^2$ as below;
$$X=span\{e^{\alpha x+\beta y} : \ -\infty<\alpha, \beta<\infty\},$$
and for every $s,t\geq0,$ define $T(s,t)$ on $X$ by,
$$(T(s,t)f)(x,y)=f(x+s,y+t),$$
which $f\in X.$

Then $\{T(s,t)\}$ is a $C_0$--semigroup of contractions on $X$. Its infinitesimal generator $(A_1, A_2)$ has the domains
$D(A_1)$ and $D(A_2)$ respectively which,
$$D(A_1)=\{f : \ f\in X , \ f_x \ exists\  and \ f_x \in X\},$$
$$D(A_2)=\{f : \ f\in X , \ f_y \ exists\  and \ f_y \in X\}.$$
and on $D(A_1)$ and $D(A_2)$,
$$A_1f=f_x \ \ \ \ ,\ \ \ \ A_2f=f_y,$$
such that $f_x$ and $f_y$ are derivatives on $x$ and $y$, respectively.

Hence $A_1$ and $A_2$ have the property such that $A_1A_2=A_2A_1$ on X.
\end{exam}

\begin{rem}
A dissipative operator $A$ for which $R(I-A)=X,$ is called m--dissipative.
If $A$ is dissipative so is $\mu A$ for all $\mu >0$ and therefore if $A$ is $m$--dissipative
then $R(\lambda I-A)=X$ for every $\lambda>0.$
In terms of $m$--dissipative operators the theorem 2.4 can be restated as:
A pair of densely defined operators $(A_1, A_2)$ is the infinitesimal generator of a
two--parameter $C_0$--semigroup of contractions if and only if these are $m$--dissipative with the property $A_1A_2=A_2A_1$.
\end{rem}


\vspace{8mm}

{\noindent Rasoul Abazari$^*$, Assadollah Niknam, Mahmoud Hassani }\\
{\noindent\small { Department of
Mathematics, Faculty of Sciences, Mashhad Branch, Islamic Azad University, P.O.Box 413-91735, Mashhad, Iran.}}\\
{\noindent\small $^*$Corresponding E-mail: {r.abazari@iauardabil.ac.ir, rasoolabazari@gmail.com }}
}}

\end{document}